\documentclass{article}

\usepackage{arxiv}

\usepackage[utf8]{inputenc} 
\usepackage[T1]{fontenc}    
\usepackage{hyperref}       
\usepackage{url}            
\usepackage{booktabs}       
\usepackage{amsfonts}       
\usepackage{nicefrac}       
\usepackage{microtype}      
\usepackage{lipsum}
\usepackage{amsmath}
\usepackage{breqn}

\usepackage{amssymb}
\usepackage{amsthm}
\usepackage{amsmath}
\usepackage{breqn}

\usepackage[all]{xy}
\usepackage{relsize}

\usepackage{csquotes}

\usepackage{lineno}
\usepackage{graphicx}
\usepackage{epsfig}
\usepackage{amsthm}
\usepackage{amsmath}
\usepackage{latexsym}
\usepackage{amsfonts}
\usepackage{amssymb}
\usepackage[all]{xy}
\usepackage{amsmath,amsfonts,amssymb,amscd,amsthm,xspace}
\usepackage[utf8]{inputenc}
\usepackage{relsize}
\newtheorem{theorem}{Theorem}
\newtheorem{definition}{Definition}
\newtheorem{lemma}{Lemma}
\newtheorem{remark}{Remark}

\title{Approximate solutions for Dorodnitzyn's gaseous boundary layer limit formula}

  \author{
  Carla V.~Valencia-Negrete
Department of Physics and Mathematics\\
Universidad Iberoamericana Ciudad de M\'{e}xico\\
  Prolongaci\'{o}n Paseo de la Reforma 880,  Mexico City ---01219, MEXICO \\
  \texttt{carla.valencia@ibero.mx} 
}

\begin{document}
\maketitle

\begin{abstract}
Oleinik's \emph{no back-flow} condition ensures the existence and uniqueness of 
solutions for the Prandtl equations in a rectangular domain $R\subset \mathbb{R}^2$.
It also allowed us to find a limit formula
for Dorodnitzyn's stationary compre\-ssible boundary layer with constant total energy
on a bounded convex domain in the plane $\mathbb{R}^2$.
Under the same assumption, we can give an approximate solution $u$ for the limit formula
if $|u|<\!\!<\!\!<1$ such that:
\[u(z)\cong \delta * c *
\left[z+\frac{6}{25}\cdot \frac{1}{2i_0} \cdot \frac{4U^2}{3}z^4\right]+o(z^5),\]
that corresponds to an approximate horizontal velocity component
when a small parameter $\epsilon$ given by the quotient of the maximum height
of the domain divided by its length tends to zero.
Here, $c>0$, $\delta$ is the boundary layer's height in Dorodnitzyn's coordinates,
$U$ is the \emph{free-stream} velocity at the upper boundary of the domain, and 
$T_0$ is the absolute surface temperature.
\end{abstract}


\section{Introduction}
\label{sec:1}

First, the limit formula is rewritten as a non-linear 
ordinary differential equation of order 1 in Lemma \ref{lem:1},
\[\frac{\partial u}{\partial s}
=c\left[1-\left(u^2/2i_0\right)\right]^{-6/25},\]
where  $i_0=c_p \ T_0$ the constant total energy value,
$c_p$ is \emph{specific heat at cons\-tant}
\emph{pressure}, $T_0$ is the absolute surface temperature,
$c=\left. \partial u/\partial s\right|_{s=0}>0$ is a strictly positive constant
that represents the \emph{no back-flow} condition continuous extension to the lower boundary,
and the fractional exponent $-6/25=19/25-1$ comes from the 
empirical \emph{Power-Law} 
$\mu/\mu_h  =   \left( T/T_h\right)^{\frac{19}{25}}$
for the absolute temperature $T$ and
the dynamic viscosity $\mu$ 
with the \emph{free-stream dynamic viscosity} $\mu_h>0$,
and the \emph{free-stream temperature} $T_h>0$
at the upper boundary \cite[p. 46]{SmitsDussauge2006}.

Then, it is shown in Lemma \ref{lem:2} that if $|u|<\!\!<\!\!<1$, 
the non-linear term can be expressed as a power series 
$\left[1-\left(u^2/2i_0\right)\right]^{-6/25}=
    1+(6/25)(1/2i_0)\ u^2
    +o(u^4)$.
In Theorem~\ref{theo:1}, we seek a solution of the form of a fourth degree polynomial,
$Az+Bz^2+Cz^3+Dz^4$, in the same way it is done in \cite{Dorod42,Lees50}.
 The existence and unicity of
the solutions comes from the Darboux sums limit of the Riemann integral 
taken to solve Eq. (\ref{eq:2}).


\section{Problem Statement}
\label{sec:3}

The limit formula  (Eq. (\ref{eq:1}) in the next page) is presented as a non-linear
ordinary differential equation of order $1$.
\begin{definition}\label{defi:1}
Let $\tilde{D}\subset \mathbb{R}^3$ be a bounded open subset 
of the three dimensional Euclidean space such that its projection
$\tilde{D}\ \cap \ \mathbb{R}^2$ 
is a planar region $D$,
\begin{equation*}
D\colon = \left\{(x,y)\in \mathbb{R}^2 \ | \ 
0<x<L \hspace{2pt} \& \hspace{4pt} 0<y<h(x)\right\},
\end{equation*}
where $L>\!\!>\!\!>h>0$, and $h\colon [0,L] \to (0,\infty)\in C^2([0,L])$ 
has derivatives up to order two in $(0,L)$
with extension to the corresponding extremes, $\{0,L\}$, of this interval, 
and $h(0)=h(L)=\delta$. 
Moreover, the upper boundary $h$ has only one critical point $c$ which 
is a global maximum, $h(c)=H\geq h(x)$  $\forall x \in [0,L]$.
We denote the topological boundary of $D$ by $\partial D$.
\end{definition}

\begin{definition}\label{defi:2}
Consider:
$(i)$ the first coordinate $\ell$ of Dorodnitzyn's diffeormorphism 
such that for all $\left(x,y\right) \in D$, 
$\ell \left(x,y\right) = 
   c_1 \hspace{2pt}  \left[1-\left(U^2/2i_0\right)\right] \hspace{2pt} x$,
where the \emph{free-stream} velocity $u|_{y=h(x)}=U>0$,
$c_1= p_0 \hspace{2pt} T_0^{2b/b-1}$, $p|_{y=0}=p_0$ is the pressure
at the lower boundary, $T_0$ is absolute surface temperature, and
$b= 1.405$ is the exponent of the \emph{adiabatic} and \emph{polytropic} atmosphere constant,
$p\ V^{b}=cte.$, for the volume $V=\int\!\!\int\!\!\int \hat{D}d\mathbf{x}$ \cite[p. 35]{Tiet};
 $(ii)$ the second coordinate $s$ of Dorodnitzyn's diffeomorphism
 \begin{equation*}
\begin{split}
  s \left(x,\hat{y}\right) & = \int_0^{\hat{y}}\rho(x,y)dy,
\end{split} 
\end{equation*}
where $\rho$ is the density; and $(iii)$ the upper boundary 
in the Dorodnitzyn's domain,
 $\delta(x,h(x))=\int_0^{h(x)}\rho(x,y)dy$.
 Then, $\mathbf{s}=(\ell,s)$ is a diffeomorphism defined in the domain $D$ \cite{val18}.
\end{definition}

\begin{lemma}\label{lem:1}
Let $D$ be a convex domain and $\mathbf{s}:D\to \mathbf{s}(D)$ as described in Definition \ref{defi:1}
and \ref{defi:2}.
Assume that $(i)$ $u\in C^2(D)$, $(ii)$
$\partial u/ \partial s>0$ in $D$ with a continuous extension to
 $\partial u/\partial s|_{s=0}=c>0$
to a positive constant $c$,
$(iii)$  $u(\ell,0)=0$ for all $\ell \in [0,1]$,
and $(iv)$ the following relation:
\begin{eqnarray}\label{eq:1}
   f
    \hspace{2pt}
    \frac{\partial^2 u}{\partial s^2}
&=&
\frac{\partial f}{\partial s}
\hspace{2pt}
\frac{\partial u}{\partial s},
\end{eqnarray}
where $f=\left[1-\left(u^2/2i_0\right)\right]^{-6/25}>0$, 
and $i_0$ is an strictly positive constant.
Then, $u$ is a solution of Eq. (\ref{eq:1}) if and only if 
it verifies Eq. (\ref{eq:2}):
\begin{eqnarray}\label{eq:2}
    \frac{\partial u}{\partial s}
&=& c
\left[1-\left(u^2/2i_0\right)\right]^{-6/25}.
\end{eqnarray}
\end{lemma}

\begin{proof}
If $g=\partial u/ \partial s$, then Eq. (\ref{eq:1}) becomes $(\partial g/\partial s)/g=(\partial f/\partial s)/f$.
This is, 
$\int_0^y\frac{\partial}{\partial s}\left[ln\left(g(s)\right)\right]ds=
\int_0^y\frac{\partial}{\partial s}\left[ln\left(f(s)\right)\right]ds$.
Thus, $exp\left[ln\left(g(y)/g(0)\right)\right]=
exp\left[ln\left(f(y)/f(0)\right)\right]$.
Therefore, 
$g(y)=\left(g\left(0\right)/f\left(0\right)\right)f(y)$.
The \emph{no-slip} condition $u(x,0)=0 $ for all $x\in[0,L]$
implies that $f(0)=1$. Finally, we substitute $g=\partial u/\partial s$, and we obtain Eq. (\ref{eq:2}).
\end{proof}


\section{Analytic solutions for the limit formula}\label{s5}

In Lemma \ref{lem:2}, we show that there is an  expression of 
the non-linear term $\left[1-\left(u^2/2i_0\right)\right]^{-6/25}$
as a convergent power series if $|u|<\!\!<\!\!<1$.
Finally, following \cite{Dorod42,Lees50}, in Theorem \ref{theo:1}, 
we may seek a solution of the form
of a fourth degree polynomial with coefficients chosen so that
the boundary conditions of the problem are satisfied.

\begin{lemma}\label{lem:2}
Let $|u|<\!\!<\!\!<1$. Then,
\begin{equation}\label{eq:3}
    \left[1-\left(u^2/2i_0\right)\right]^{-6/25}=
    1+\frac{6}{25}\cdot \frac{u^2}{2i_0}
    +\frac{6}{25}\cdot \frac{u^4}{2(2i_0)^2}
    +\frac{1}{2}\cdot\frac{6}{25}\cdot \frac{u^4}{(2i_0)^2}+o(u^6).
\end{equation}
\end{lemma}

\begin{proof}
If $\sigma=1-(u^2/2i_0)\cong 1$, then $\theta =(-6/25)\ln
    \left(1-(u^2/2i_0)\right)<\!\!<\!\!<1 $. We apply 
the \emph{Maclaurin formula} for the exponential,
$e^{\theta}=1+\theta+(1/2)\theta^2+(1/3!)\theta^3+\cdot \cdot \cdot$,
and the Taylor series expansion for the logarithmic function if $|\sigma|<1$, $log(1+\sigma)=\sigma-(1/2)\sigma^2+(1/3!)\sigma^3+\cdot \cdot \cdot$, so that:
\begin{equation*}
\begin{split}
 \left[1-\left(u^2/2i_0\right)\right]^{-6/25}&=
    \exp\left[\frac{-6}{25}\ \ln
    \left(1-\frac{u^2}{2i_0}\right)\right],\\
    &=
    \exp\left[\frac{-6}{25}\ \left\{\ln
    \left(1+\frac{u}{\sqrt{2i_0}}\right)+\ln
    \left(1-\frac{u}{\sqrt{2i_0}}\right)\right\}\right],\\
    &=
    1+\frac{6}{25}\cdot \frac{u^2}{2i_0}
    +\frac{6}{25}\cdot \frac{u^4}{2(2i_0)^2}
    +\frac{1}{2}\cdot\frac{6}{25}\cdot \frac{u^4}{(2i_0)^2}+\\
    & \quad +\frac{6}{25}\cdot \frac{u^6}{3(2i_0)^3}+\frac{6}{25}\cdot \frac{u^6}{2(2i_0)^3}+o(u^8).  
\end{split}
\end{equation*}
\end{proof}

\begin{remark}
In order to identify the coefficients of a polynomial expression for $u$,
we observe Dorodnityzn's steps \cite{Dorod42,Lees50}.
If we take the upper boundary conditions for the second order derivative from the pro\-blem stated by Dorodnitzyn on the domain $\mathbf{s}(D)$, then the boundary conditions establish a system of four equations for the coefficients $A$, $B$, $C$, $D$
so that $u$ can be expressed as a fourth degree polynomial $Az+Bz^2+Cz^3+Dz^4$ for $z=s/\delta$
if $A=U\left[2+\left(\lambda/6\right)\right]$, $B=-U\left(\lambda/2\right)$, $C=U\left[\left(\lambda/2\right)-2\right]$, and $D=U\left[1-\left(\lambda/6\right)\right]$,
where 
$\lambda$ is the Pohlhausen coefficient
$\lambda=\left[-\delta/\left(1-U^2\right)\right]\partial U/\partial \ell$,
and $\lambda=0$ if $U$ is constant.
\end{remark}

\begin{remark}
For each fixed value of $x\in [0,L]$, the height's normalization $z(x)=s(x)/\delta(x) \in [0,1]$.
Additionally, $\partial u/\partial z=\delta * \partial u/\partial s$ because
$\partial u/\partial z=(\partial u/\partial s)(\partial s/\partial z)+(\partial u/\partial \ell)(\partial \ell/\partial z)$ and $\ell$ is independent of $z$.
\end{remark}

\begin{theorem}\label{theo:1}
Under the same conditions of Lemma \ref{lem:1}, Eq. (\ref{eq:2})
has a unique solution
\[u(z) =  \delta * c *
\left[z+\frac{6}{25}\cdot \frac{1}{2i_0} \cdot \frac{4U^2}{3}z^4\right]+o(z^5),\]
for the normalized height $z=s/\delta \in [0,1]$.
\end{theorem}

\begin{proof}
Given Eq. (\ref{eq:2}), we can substitute Eq. (\ref{eq:3}), so that:
\begin{equation}\label{eq:5}
\begin{split}
    \frac{\partial u}{\partial s}
&= \frac{1}{\delta} \frac{\partial u}{\partial z},\\
&= c
\left[1-\left(u^2/2i_0\right)\right]^{-6/25},\\
&= c \left[1+\frac{6}{25}\cdot \frac{u^2}{2i_0}
    +\frac{6}{25}\cdot \frac{u^4}{2(2i_0)^2}
    +\frac{1}{2}\cdot\frac{6}{25}\cdot \frac{u^4}{(2i_0)^2}+o(u^6)\right].\\
\end{split}
\end{equation}
Because the \emph{stream-flow velocity} $U$ is constant, then the Pohlhausen coefficient $\lambda=0$.
Thus, we can seek a solution with a fourth degree polynomial form $u(z)=2Uz-2Uz^3+Uz^4$.
If we replace this expression in the Eq. (\ref{eq:5}):
\begin{equation*}
\begin{split}
    u(z)
&= \delta * c * \int_0^z 
\left[1+\frac{6}{25}\cdot \frac{u(\tau)^2}{2i_0}
    +\frac{6}{25}\cdot \frac{u(\tau)^4}{2(2i_0)^2}
    +\frac{1}{2}\cdot\frac{6}{25}\cdot \frac{u(\tau)^4}{(2i_0)^2}+o(u^6)\right] d\tau,\\
&  = \delta * c *
\left[z+\frac{6}{25}\cdot \frac{1}{2i_0} \cdot \frac{4U^2}{3}z^4\right]+o(z^5).\\
\end{split}
\end{equation*}
\end{proof}

\section{Conclusion}\label{s6}

Olga Oleinik's \emph{no back-flow} condition is sufficient to show the existence
and unicity of solutions for Dorodnitzyn's gaseous boundary layer limit
when the horizontal velocity component is small.
The fact that one can find analytic approximations for in a compressible 
boundary layer when the horizonal component of the velocity is small allows us 
not only to empirically see that the velocity profile is stable at low velocities,
but also to analytically understand that, in this case, 
one can arrive at flow expressions that linearly depend on the height.




\section{Appendix}
\label{sec:sample:appendix}

Funding: This research did not receive any specific grant from funding agencies in the public, commercial, or not-for-profit sectors.

Declarations of interest: none

\bibliographystyle{unsrt}  


\end{document}